\documentclass[a4paper,11pt]{article}
\usepackage[utf8]{inputenc}
\usepackage{CommandsAndStyle}

\usepackage[colorinlistoftodos,bordercolor=orange,backgroundcolor=orange!20,linecolor=orange,textsize=scriptsize]{todonotes}

\usepackage[ruled, vlined, norelsize, noresetcount]{algorithm2e}

\SetAlgorithmName{Algorithm RSHT}{rsht}{rsht}

\usepackage{subcaption}
\usepackage{booktabs}
\usepackage{paralist}
\usepackage{fixltx2e}
\usepackage{pgf,tikz}
\usepackage{mathrsfs}
\usetikzlibrary{arrows}

\usepackage{mathtools}
\usepackage{amsmath}
\usepackage{hyperref}

\usepackage[ruled, vlined, norelsize, noresetcount]{algorithm2e}

 \newcommand{\DoTikzmark}[1]{%
  \tikz[remember picture] \coordinate[shift={(0,.7ex)}](#1);%
}

\newcommand{\DoTikzmarkcol}[1]{%
  \tikz[remember picture] \coordinate[shift={(.6ex,.7ex)}](#1);%
}

\newcommand{\colrow}[3][]{%
  \tikz[overlay,remember picture, line width=12pt]
    \draw[shorten >=-.1em, shorten <=-.1em, #1] (#2)--(#3);
}

\newcommand{\colcol}[3][]{%
  \tikz[overlay,remember picture, line width=12pt]
    \draw[shorten >=-.5em, shorten <=-.5em, #1] (#2)--(#3);
}

\newcommand{\h}{\mbox{\rm H}}

\title{Hadamard Matrix Torsion}
\author{Davide Lofano, Frank H.~Lutz}
\date{November 3, 2021}

\begin{document}

\maketitle

\begin{abstract}
We construct a series HMT$(n)$ of $2$-dimensional simplicial complexes with torsion 
$H_1($HMT$(n))=(\Z_2)^{{k}\choose{1}} \times (\Z_4)^{{k}\choose{2}} \times \cdots \times (\Z_{2^k})^{{k}\choose{k}}$,
$|H_1($HMT$(n))|=|$det(H$(n))|=n^{n/2} \in \Theta(2^{n \log n})$, where the construction is based 
on the Hadamard matrices $\h(n)$ for $n\geq 2$ a power of $2$, i.e., $n=2^k, \ k \geq 1$. 
The examples have linearly many vertices, their face vector is $f(\mbox{\rm HMT}(n))=(5n-1,3n^2+9n-6,3n^2+4n-4)$.

Our explicit series with torsion growth in $\Theta(2^{n \log n})$ is constructed in quadratic time $\Theta(n^{2})$ 
and improves a previous construction by Speyer \cite{Speyer:blog} with torsion growth in $\Theta(2^{n})$, 
narrowing the gap to the highest possible asymptotic torsion growth in $\Theta(2^{n^2})$ 
proved by Kalai~\cite{Kalai1983} via a probabilistic argument.
\end{abstract}

\section{Introduction}

The most elementary way to build a two-dimensional CW complex with torsion $\Z_r$ 
in the first integer homology starts out with a single polygonal disc with $r$ edges, $r\geq 2$, 
that are all oriented in the same direction and jointly identified. The resulting CW complex
has one vertex, one edge, and one two-dimensional face---it has homology $H_*=(\Z,\Z_r,0)$.

\begin{definition}[Matrix Disc Complexes]\label{Def:DCM}
     Let $M =(M_{ij})$ be an $(m\times n)$-matrix with integer entries.
    Let the set of two-dimensional \textbf{matrix disc complexes} $DC(M)$ associated with $M$
    comprise the CW complexes constructed level-wise in the following way:
    
    \begin{itemize}
     \item Every complex in $DC(M)$ has a single $0$-cell (with label $0$ in the following).
     
     \item The $1$-skeleton of a complex in $DC(M)$ has an edge cycle $a_j$ for every column index $j \in \{1, \ldots ,n\}$ of the matrix $M$.
     
     \item Every row $i$ of $M$ with row sum $s_i = |M_{i1}| + \ldots + |M_{in}|$, $i \in \{1, \ldots ,m\}$, contributes a polygonal disc with $s_i$ edges. 
              For every positive entry $M_{ij}$, ${M_{ij}}$ edges of the disc are oriented coherently and are assigned with the label $a_j$. 
              In the case of a negative entry, the direction of the corresponding edges is reversed; in the case of a zero-entry, 
              the respective edge does not occur.
    \end{itemize}

\end{definition}

\begin{example}
The $(1\times 1)$-matrix $M=(r)$ yields a single disk with the $r$ identified edges all oriented in the same direction---the elementary construction from above.
\end{example}

\begin{example}
\label{Ex:klein_rp2}
Let us now consider $M=(2 \ 2)$ with one row. The examples in DC($M$) have two cycles, $a_1$ and $a_2$, and a single disc with edges $2a_1+2a_2$. 
There are two choices for the edge sequences of the disc, either $a_1 a_1 a_2 a_2$ or $a_1 a_2 a_1 a_2$. In the first case, the resulting CW complex 
is the Klein bottle (Figure \ref{Fig:2}, left), in the second case, we obtain a pinched real projective plane $\mathbb{R}P^2$ (Figure \ref{Fig:2}, right). 
While the Klein bottle is a manifold, the latter example is not---thus the two examples in $DC(M)$ are not homeomorphic to each other. 

\begin{figure}[t]
\begin{center}
\includegraphics[width=0.6\columnwidth]{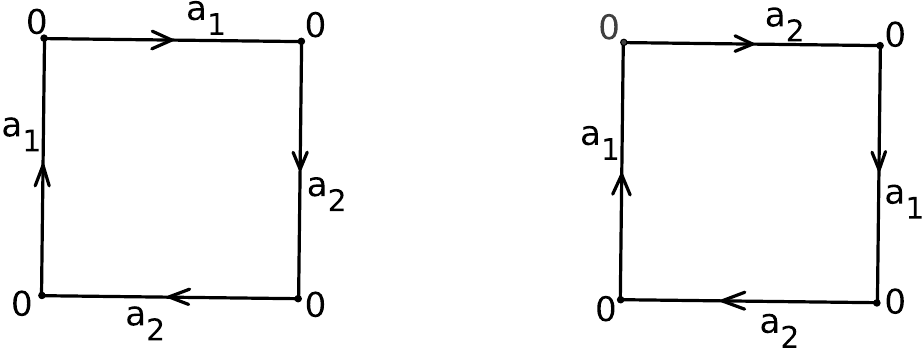}
\end{center}
\caption{Klein bottle on the left and a pinched $\mathbb{R}P^2$ on the right, both obtained from the matrix $M=(2 \ 2)$.}
\label{Fig:2}
\end{figure}
\end{example}

\begin{lemma}\label{Lemma:Hom}
 The examples in $DC(M)$ all have the same integer homology $H_*$. 
\end{lemma}

\begin{proof}
The examples $DC(M)$ all have a single vertex only and therefore are connected, respectively.
For every representative $C \in DC(M)$ it thus follows that $H_0(C)=\mathbb Z$.
Each edge of $C$ is a cycle, i.e., the first homology group $H_1(C)$ of $C$ 
is determined by the $n$ rows of $M$ as relations, and therefore $H_1$ coincides for all the examples in $DC(M)$.
Further, the second homology $H_2$ of any representative $C$ is simply the kernel of the matrix~$M$.
\end{proof}

\begin{example}[Projective plane]\label{Ex:RP2}
 
For the $(2\times 2)$-matrix\, $\h(2)=
\begin{pmatrix*}[r]
 1 & 1 \\  1 & -1
\end{pmatrix*}
$\, 
we have two discs with edges $a_1a_2$ and $a_1a_2^{-1}$, as in Figure \ref{Fig:1} on the left. 
If we glue together the two discs along the common edge $a_2$, as in Figure \ref{Fig:1} on the right, 
we obtain a single disc with two (identified) edges $a_1a_1$, the standard scheme for the 
real projective plane $\mathbb{R}P^2$ with homology $H_*=(\Z,\Z_2,0)$.
           
\begin{figure}[t]
\begin{center}
\includegraphics[width=0.73\columnwidth]{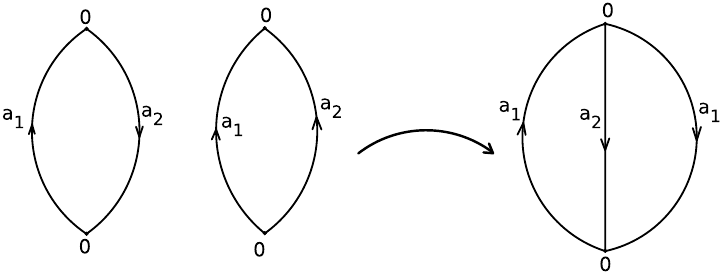}
\end{center}
\caption{$\mathbb{R}P^2$ as a CW complex with two disks.}
\label{Fig:1}
\end{figure}

\end{example}

A simple consequence of the proof of Lemma~\ref{Lemma:Hom} is that if $M$ is a square matrix with $\mbox{det}(M)\neq 0$, 
then $|H_1(C)|=|\mbox{det}(M)|$ for any example $C \in DC(M)$. Our goal in the following is to construct triangulations 
of CW-complexes $C \in DC(M)$ as abstract simplicial complexes with few vertices, in particular, for square matrices $M$ 
with large determinant, yielding simplicial complexes with huge torsion.

Kalai proved in \cite{Kalai1983} that, asymptotically, there are $\mathbb{Q}$-acyclic simplicial complexes on $n$ vertices with torsion growth in $\Theta(2^{n^2})$, 
and that this is the maximal possible growth. Recently, Newman \cite{newman2019} showed, via a randomized construction, 
that any abelian group $G$ can be obtained as torsion of a simplicial complex with $\Theta(\log (|G|)^{\frac{1}{2}})$ vertices.
Explicit classes of $\mathbb Q$-acyclic simplicial complexes were provided by Linial, Meshulam and Rosenthal~\cite{Linial2010}, 
called sum complexes, however, without control on the torsion growth.

The overall approach in this paper is inspired by a construction of Speyer \cite{Speyer:blog} of $2$-di\-men\-sional $\mathbb{Q}$-acyclic 
simplicial complexes on $\Theta(n)$ vertices with exponential torsion growth $\Theta(2^{n})$, corresponding to particular square matrices $M$
of size $\Theta(n)\times \Theta(n)$; see Section~\ref{Subsec:Speyer}.
For general matrices $M$, a first (elementary) triangulation approach for associated complexes $C \in DC(M)$ is given in Section~\ref{Subsec:triangle}.

Starting with Hadamard matrices $\h(n)$, we give an explicit triangulation construction in quadratic time $\Theta(n^2)$ 
that achieves torsion growth $\Theta(2^{n \log n})$; see Section~\ref{Sec:improved}.

\section{Preliminaries}
\label{Sec:prelim}

In this section, we first  give a (straight-forward) procedure to triangulate  a general complex $C \in DC(M)$
for some given integer $(m \times n)$-matrix $M=(M_{ij})$. Afterwards, we discuss basic facts about Hadamard matrices $\h(n)$ 
that will be used in the main Section~\ref{Sec:improved}, where we introduce an improved triangulation scheme
for the complexes corresponding to these square matrices to obtain our torsion growth bound $\Theta(2^{n \log n})$.

\subsection{A first triangulation procedure}
\label{Subsec:triangle}

Given some $(m \times n)$-matrix $M=(M_{ij})$, $i \in \{1, \ldots , m\}$, $j \in \{1, \ldots , n\}$, the complexes $C \in DC(M)$ are $2$-dimensional CW complexes 
with a single vertex (with label $0$), $m$~cycles $a_j, \ j \in \{1, \ldots , n\}$, corresponding to the column indices,
and a disc for every non-zero row of $M$. In the following, we assume that $M$ has no zero columns and rows.
(In case rows appear multiple times, we saw by Example~\ref{Ex:klein_rp2} that different choices for the edge sequences
can yield different gluings of the discs.)

When we subdivide a CW complex $C$ to obtain a triangulation of it as an abstract simplicial complex $K$,
we have to ensure that $K$ has no loops and no parallel edges:

\begin{itemize}

 \item The only vertex $0$ of $C$ is kept as vertex $0$ in $K$.

 \item Each cycle $a_j, \ j \in \{1, \ldots , n\}$, is triangulated by using two extra vertices, $v^1_j$~and~$v^2_j$.
 
 \item We next triangulate the $m$ discs corresponding to the $m$ rows of $M$.
 For each row $i \in \{1, \ldots , m\}$, the respective disc is an $s_i$-gons for  $s_i = |M_{i1}| + \ldots + |M_{in}|$.
 
 Into each $s_i$-gon we place 
 a $\lceil \frac{3}{2}s_i\rceil$-gon using $\lceil \frac{3}{2}s_i\rceil$ new vertices $c_i^k, \ k \in \{1, \ldots , \lceil \frac{3}{2}s_i\rceil\}$. 
 The inside of the $\lceil \frac{3}{2}s_i\rceil$-gon can be triangulated arbitrarily  by consecutively adding diagonals. 
 The annulus between the inner $\lceil \frac{3}{2}s_i\rceil$-gon and the outer $s_i$-gon is triangulated by connecting 
 any inner vertex with three consecutive vertices of the $s_i$-gon, creating a cone, and then filling the remaining gaps with triangles. 
 In particular, we choose a starting vertex on the $s_i$-gon and a direction, and we connect $c_i^1$ with the starting vertex 
 and the two consecutive ones. Next, we connect $c_i^2$ with the last vertex to which we connected $c_i^1$ 
 and then to the consecutive two vertices. We continue like this till we reach the last vertex of the $\lceil \frac{3}{2}s_i\rceil$-gon 
 which will be connected with the starting vertex again (be careful, this last vertex could be connected with only two vertices 
 of the outer $s_i$-gon instead of three). See Figure~\ref{Fig:3} for an example on how to fill the inside. 
 
\end{itemize}

\begin{remark}\label{Rem:1}
 In many examples, we could actually triangulate the polygonal discs with fewer vertices than according to the above procedure. 
 However, improvements will depend on the concrete entries of the given matrix.
\end{remark}

Since all the polygonal discs are triangulated in a way such that no two discs share an interior edge we have that the homology 
of the constructed simplicial complex $K$ is the same as that of the original CW complex $C$.

In the following proposition on triangulations of matrix disc complexes we use the Smith Normal Form of a matrix $M$,
so before the proposition, let us remember the definition.

\begin{definition}(Smith \cite{smith1861})
 Let $M$ be a nonzero $(m \times n)$-matrix over a principal ideal domain. There exist invertible $(m \times m)$- and $(n \times n)$-matrices 
 $S$ and $T$, respectively, such that $M=SAT$ and 
 \begin{equation*}
 A=
\begin{pmatrix*}[r]
    \alpha_1 & 0 & & \dots &  &  & 0 \\
    0 & \alpha_2 &0 && \dots   & & 0\\
    \vdots & 0 & \ddots  & 0 & \dots & 0 & \vdots \\
     & \vdots &0 & \alpha_r & & &  \\
     &  & &  & 0 &  &  \\
     &  & &  &  & \ddots  &  \\
    0 & 0 & & \dots &  &  &  0\\
\end{pmatrix*},
\end{equation*}
where $\alpha_i \, | \, \alpha_{i+1}$ for all $1 \leq i < r$. $A$ is called the \emph{Smith Normal Form} of $M$.
\end{definition}

 \begin{proposition}\label{Prop:1}
  Given an $(m \times n)$-matrix $M$ and $A=(\alpha_i)$ its Smith Normal Form, 
  there is a $2$-dimensional simplicial complex $K$ on $V(K)$ vertices with
  \begin{equation*}
   V(K) \leq 2n+m+1+\frac{3}{2}\sum_{i,j}{| M_{ij} |}.
  \end{equation*}
  Furthermore,
  \begin{equation*}
   H_1(K)=\Z/\alpha_1 \Z \times \cdots \times \Z/\alpha_r \Z \times \Z^{n-r}.
  \end{equation*}

 \end{proposition}

 \begin{proof}
 Let $M$ be an $(m \times n)$-matrix, $C \in$ DC($M$), $K$ a triangulation of $C$ as described above, and $s_i = |M_{i1}| + \dots + |M_{in}|$. We easily see that  
 \begin{equation*}
  \sum_{i=1}^{m}{\left\lceil \frac{3}{2}s_i\right\rceil} \leq m+\frac{3}{2}\sum_{i,j}{|M_{ij} |}.
 \end{equation*}
 For the number of vertices $V(K)$ of $K$ we obtain the bound
 \begin{equation*}
   V(K)=1+2n+\sum_{i=1}^{m}{\left\lceil\frac{3}{2}s_i\right\rceil}\leq 2n+m+1+\frac{3}{2}\sum_{i,j}{|M_{ij} |},
  \end{equation*}
with $1$ vertex for the original vertex of $C$, $2n$ vertices to triangulate the $n$ cycles, and $\sum_{i=1}^{m}{\left\lceil \frac{3}{2}s_i\right\rceil}$
interior vertices in the $m$ discs. Since $M$ is the boundary matrix for the homology of $K$, the homology of $K$ is represented by the Smith Normal Form $A$ of~$M$.
 \end{proof}

\begin{remark}
 Given a square matrix $M$ with det$(M) \neq 0$, it follows by Proposition~\ref{Prop:1} that for the simplicial complex $K$ associated with $M$, $|H_1(K)|=|$det$(M)|$.  
\end{remark}

\subsection{Speyer's construction}
\label{Subsec:Speyer}
 
Speyer \cite{Speyer:blog}  provided a construction to produce $2$-dimensional $\mathbb{Q}$-acyclic simplicial complexes 
for which the size of the torsion grows exponentially in the number of vertices.

Let $k \geq 2$ be an integer and $k=\gamma_m 2^m+\gamma_{m-1} 2^{m-1} + \ldots +\gamma_0 2^0$ be its binary expansion, 
with leading coefficient $\gamma_m=1$ and otherwise $\gamma_i \in \{0, 1\}$ for all $ 0 \leq i \leq m-1$. 
An $((m+1) \times (m+1))$-matrix $M(k)$ is constructed in the following way:
\begin{itemize}
 \item The first row contains the entries $(-1)^i \gamma_{m-i}$ for $ i \in \{0, \ldots ,m \}$.
 
 \item The lower part of the matrix $M(k)$ is an $(m \times (m+1))$-matrix with $1$'s on the first diagonal followed by $2$'s 
 on the diagonal to the right, and all other entries equal to zero. It is then easy to see that det$(M(k))=k$.
\end{itemize}

Using Proposition \ref{Prop:1}, there is a $2$-dimensional simplicial complex $K$ on $V(K) \leq 3(m+1)+1+\frac{3}{2}(3m+(m+1))\leq 9m+6$ vertices 
corresponding to $M(k)$ that has torsion of size $k$. Since the number of vertices is linear in $m$, the torsion $k$ grows exponentially 
in the size $n=m+1$ of the matrix, i.e., $k\in\Theta(2^n)$.

\begin{example}\label{Ex:Speyer}
For an explicit example, we consider $k=11=8+2+1$ with
\[M(11)=
 \begin{pmatrix*}[r]
\DoTikzmark{num1}1 & 0 & 1 & -1 \DoTikzmark{num2} \\
1 & 2 & 0 & 0  \\
0 & 1 & 2 & 0  \\
0 & 0 & 1  &  2
\end{pmatrix*}.
\]
\colrow[black ,opacity=.2]{num1}{num2}

\noindent
Figure \ref{Fig:3} displays a triangulation $K(11)$ of a representative $C \in$ DC($M(11)$) with $29$ vertices
that has torsion $\Z_{11}$. 

\begin{remark}
In case $k$ is prime, the torsion of complexes corresponding to the matrix $M(k)$ is cyclic, but in general it is a product of the factors
in the Smith Normal Form of~$M(k)$.
\end{remark}

As already pointed out in Remark \ref{Rem:1}, also in this particular example  $K(11)$  we could save on the number of vertices 
necessary for the triangulation of the interiors of the polygonal discs. However, it is open what the minimum number 
of vertices is for a re-triangulation of $K(11)$.

\begin{figure}[t]
\begin{center}
\includegraphics[width=0.8\columnwidth]{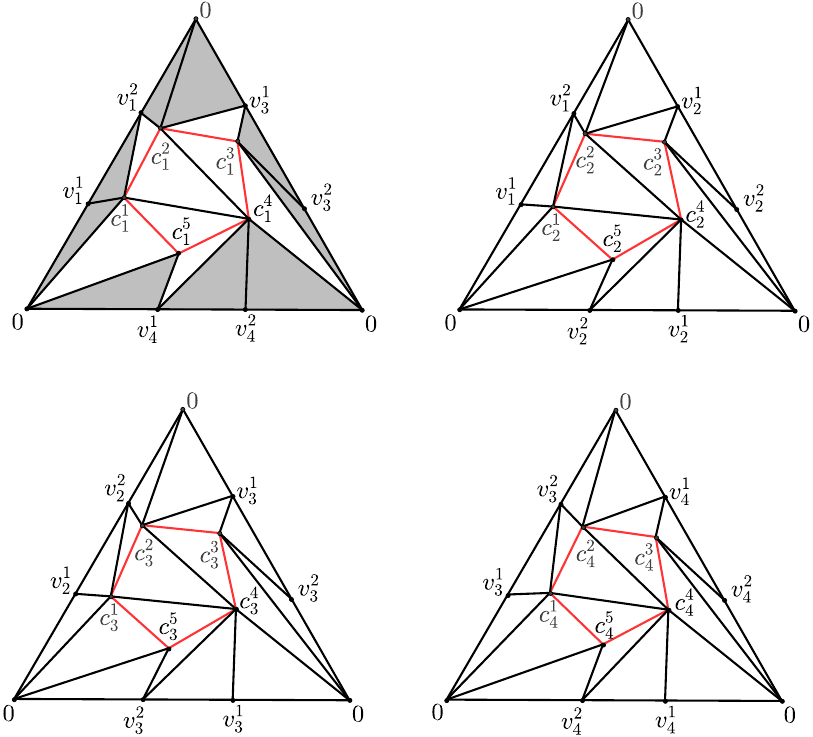}
\end{center}
\caption{The four subdivided triangles of the complex $K(11)$ of Example~\ref{Ex:Speyer}. 
In the upper left part of $K(11)$ we highlighted the five cones of inner vertices with respect to the subsequences of three boundary vertices.}
\label{Fig:3}
\end{figure}

\end{example}

\enlargethispage*{5mm}

\subsection{Hadamard matrices}
 
In our quest to construct $2$-dimensional simplicial complexes with few vertices but high torsion via matrices, Hadamard matrices are of particular interest.
In fact, Hadamard matrices were used earlier in geometry for extremal constructions; see for example \cite{HUDELSON} and \cite{ziegler2000lectures}. 
 
\begin{definition}
 A Hadamard matrix $\h$ is a square $(n \times n)$-matrix whose entries are either $+1$ or $-1$ and whose rows are mutually orthogonal.
\end{definition}

Hadamard matrices $\h$ have in common that their determinants attain the Hadamard bound, $|\mbox{\rm det}(\h)|=n^{\frac{n}{2}}$. 
In general, $|\mbox{\rm det}(M)|\leq n^{\frac{n}{2}}$ for any integer $(n \times n)$-matrix 
\mbox{$M=(M_{ij})$} with $|M_{ij}| \leq 1$, and a matrix $M$ attains the bound if and only if it is a Hadamard matrix \cite{hadamard1893}.

 It is known that the order of a Hadamard matrix must be $1$, $2$, or a multiple of $4$, but it is open whether Hadamard matrices exist for all multiples of 4. 
  However, a very nice construction of Sylvester  \cite{sylvester1867}  tells us that if $\h$ is a Hadamard matrix of order $n$, then the matrix
 $
 \begin{pmatrix*}[r]
  \h & \h \\
  \h & -\h
 \end{pmatrix*}
$ is a Hadamard matrix of order $2n$.

The following sequence of matrices $\h(n), \ n=2^k, \ k \geq 1$, also called Walsh matrices \cite{walsh1923}, are Hadamard matrices:

\begin{equation*}
 \h(1)=
\begin{pmatrix*}[r]
 1 
\end{pmatrix*},
\end{equation*}

\begin{equation*}
 \h(2^k)=
\begin{pmatrix*}[r]
 \h(2^{k-1}) & \h(2^{k-1}) \\
 \h(2^{k-1}) & -\h(2^{k-1})
\end{pmatrix*}, \ \mbox{\rm for}  \ k \geq 1.
\end{equation*}
In the following, we will always refer to this sequence when we talk about Hadamard matrices. The matrix $\h(2)$
was discussed before in Example~\ref{Ex:RP2} above.

\begin{lemma}\label{Lemma:Smith}
 Let $n=2^k, \ k \geq 0$, and let A $=(\alpha_i)$ be the Smith Normal Form of\, $\h(n)$, with\, $\h(n)=SAT$\, for integral invertible matrices $S$ and $T$. 
 
 Then  $\mbox{\rm diag}(A)=(\alpha_i)=(2^0,2^1, \ldots , 2^1,\ldots , 2^j, \ldots , 2^j, \ldots, 2^{k-1}, \ldots , 2^{k-1}, 2^k)$, where each $2^j$ appears $k \choose{j}$ times.
\end{lemma}

\begin{proof}
 We prove the statement by induction, where the base case for $n=1$ is clear, since $\h(1)=(1)$.
 
 Let $A=S^{-1} \h(n) T^{-1}$ be the Smith Normal Form of $\h(n)$. For the following two invertible matrices\, 
 $\widetilde{S}=\begin{pmatrix*}[c]  S^{-1} & 0 \\ S^{-1} & -S^{-1} \end{pmatrix*}$\, and\, $\widetilde{T}=\begin{pmatrix*}[c]  T^{-1} & -T^{-1} \\ 0 & T^{-1} \end{pmatrix*}$\, we have:
 
 \begin{align*}
  \widetilde{S} \h(2n) \widetilde{T} &= \begin{pmatrix*}[c]  S^{-1} & 0 \\ S^{-1} & -S^{-1} \end{pmatrix*} 
  \begin{pmatrix*}[r] \h(n) & \h(n) \\ 
   \h(n) & -\h(n) \end{pmatrix*} \begin{pmatrix*}[c]  T^{-1} & -T^{-1} \\ 0 & T^{-1} \end{pmatrix*} \\
   & = \begin{pmatrix*}[c] S^{-1} \h(n) & S^{-1} \h(n) \\ 0 & 2 S^{-1} \h(n) \end{pmatrix*} \begin{pmatrix*}[c]  T^{-1} & -T^{-1} \\ 0 & T^{-1} \end{pmatrix*} \\
   &= \begin{pmatrix*}[c] S^{-1} \h(n) T^{-1} & 0 \\ 0 & 2 S^{-1} \h(n) T^{-1} \end{pmatrix*} = \begin{pmatrix*}[r] A & 0 \\ 0 & 2A \end{pmatrix*}.
 \end{align*}

 The resulting matrix is the Smith Normal Form of $\h(2n)$---after a suitable reordering of the diagonal elements to ensure 
 that $\alpha_i | \alpha_{i+1} \ \mbox{\rm for all } 1 \leq i \leq 2n-1$. The lemma then follows from the known relation ${{k+1} \choose {j}}= {k \choose {j-1}} + {k \choose {j}}$ of Pascal's triangle.
\end{proof}

\section{An improved triangulation procedure}
\label{Sec:improved}

Using Proposition \ref{Prop:1}, we can triangulate a complex $C \in$ DC(H$(n)$) with $\Theta(n^2)$ vertices to produce torsion of size $n^{\frac{n}{2}} \in \Theta(2^{n \log  n})$. 
Asymptotically, this is a worse bound than what can be achieved by considering the Speyer matrices $M(k)$ of Section~\ref{Subsec:Speyer} 
that yield triangulations with $\Theta(n)$ vertices and torsion of size $\Theta(2^n)$. The inferior torsion growth is due to using linearly many additional vertices 
inside each of the $n$ discs for the examples $C \in$ DC(H$(n)$), compared to the constant number of vertices needed for the discs associated 
with the rows of the lower part of the Speyer matrices $M(k)$. Our aim in this section is to provide a modified construction for complexes 
associated with the Hadamard matrices $\h(n)$ that requires in total only linearly many vertices, $5n-1$, but still yields torsion of size $n^{\frac{n}{2}}$.

\subsection{A modified CW disc construction}

Instead of triangulating complexes $C \in $ DC($\h(n)$) in the way outlined in Section \ref{Subsec:triangle}, in the following we consider CW complexes $\widetilde{C}$ 
that are derived from the matrices $\h(n)$ in a modified way, yet keeping the homotopy type, i.e. $\widetilde{C} \simeq C$. Our goal is to use exactly one interior vertex 
for each of the polygonal discs (see the square discs in the upper part of Figure \ref{Fig:4}).

The vertex $0$ appears $n$ times along the boundary of each of the polygonal discs and therefore has to be shielded off when we triangulate the interiors 
of the discs---to avoid unwanted identifications of interior edges of the discs.

The easiest way to shield off the special vertex $0$ is by connecting its two neighbors, at every occurence along the boundary of the discs, by a diagonal. 
However, we have to ensure that each such diagonal is used only once for all of the discs. To meet this requirement, we modify our CW disc construction:
\begin{itemize}
 \item A complex associated with $\h(n)$ has a single $0$-cell only.
 
 \item To each column $j$ of $\h(n)$ we associate two different edge cycles $a_j^+$ and $a_j^-$.
 
 \item To each row $i$ of $\h(n)$ we associate a polygonal disc, as before in Definition \ref{Def:DCM}. In addition we consider $n$ discs that each connect
  the two cycles $a_j^+$ and $a_j^-$.
\end{itemize}

For every complex $C \in$ DC($\h(n)$), we then define an \emph{augmented} complex $\widetilde{C}$, so that in each disc of $C$ we revert negatively oriented edges 
by using the connecting digon pieces; see Figure~\ref{Fig:5}.

\begin{figure}[t]
\begin{center}
\includegraphics[width=0.5\columnwidth]{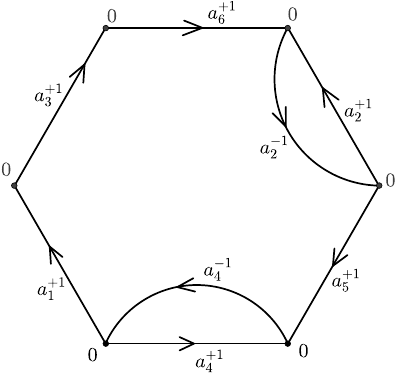}
\end{center}
\caption{A polygonal disc in the construction of $\widetilde C$.}
\label{Fig:5}
\end{figure}

The boundaries of the polygonal discs still represent the defining rows of the matrix $\h(n)$, and since the digon connectors are homotopy equivalent to single cycles, 
it follows that $\widetilde {C} \simeq C$ for the modified complexes.

Towards the interior of the discs, we have positively oriented edges only---which will allow (see Section~\ref{Subsec:valid}) 
to choose an ordering of the cycles on the boundary of each polygonal disc so that each diagonal that shields
the vertex $0$ occurs only once in $\widetilde{C}$.

\begin{remark}
An augmented complex $\widetilde{C}$ has a straight-forward representation by an augmented $(2n \times 2n)$-matrix $\widetilde{\h}(n)$ 
so that $\widetilde{C} \in$ DC$(\widetilde{\h}(n))$. Here, we obtain $\widetilde{\h}(n)$ from $\h(n)$ by splitting every column into two columns, 
where the first copy contains the positive entries of the original column and the second copy contains the absolute values 
of the negative entries of the original column. We further add $n$ new rows that represent the inserted digons.
 
 E.g., the matrix $\h(2)=\begin{pmatrix*}[r] 1 & 1 \\ 1 & -1 \end{pmatrix*}$ is augmented to
 \begin{equation*}
  \widetilde{\h}(2)=
  \begin{pmatrix*}
   \DoTikzmarkcol{num3} 1 & \DoTikzmarkcol{num5} 0 & \DoTikzmarkcol{num7} 1 & \DoTikzmarkcol{num9} 0 \\
    \DoTikzmarkcol{num4} 1  & \DoTikzmarkcol{num6} 0 & \DoTikzmarkcol{num8} 0 & \DoTikzmarkcol{num10}1  \\
   \DoTikzmark{num11} 1 & 1 \DoTikzmark{num12}& 0 & 0 \\
   0 & 0 & \DoTikzmark{num13} 1 & 1 \DoTikzmark{num14}
  \end{pmatrix*},
  \end{equation*}
 where the positive entries of $\h(2)$ are listed in the blue columns, whereas the green columns represent the negative entries of $\h(2)$. The connecting digons are highlighted in red.
 By subtracting the second column from the first one and the fourth column from the third, we obtain the matrix 
 $\begin{pmatrix*}[r] 1 & 0 & 1 & 0 \\ 1 & 0 & -1 & 1 \\ 0 & 1 & 0 & 0 \\ 0 & 0 & 0 & 1 \end{pmatrix*}$. In particular, it follows that $|\mbox{\rm det}(\widetilde {\h}(2))|=|\mbox{\rm det}(\h(2)|$.
\end{remark}

  \colcol[blue ,opacity=.25]{num3}{num4}
 \colcol[green ,opacity=.25]{num5}{num6}
  \colcol[blue ,opacity=.25]{num7}{num8}
 \colcol[green ,opacity=.25]{num9}{num10}
 \colrow[red ,opacity=.15]{num11}{num12}
 \colrow[red ,opacity=.15]{num13}{num14}

\subsection{Valid sequences}
\label{Subsec:valid}

To avoid identified diagonals, we next choose suitable orderings of the edge boundaries of the discs to select representatives $\widetilde{C}(n)\in DC(\widetilde{\h}(n))$ 
that can be triangulated with linearly many vertices. In particular, we ensure that two consecutive cycles occur exactly once along the boundaries of the discs.

\begin{definition}\label{Def:valid}
 Let $M=(M_{ij})$ be an $(n \times n)$-matrix with $\pm 1$-entries. A \emph{valid sequence} $(\tau_i^j)$
 is a sequence of orderings of the positively oriented $n$ boundary edges $a_j^{\pm}$, $j \in \{1,\ldots,n\}$, of the first $n$ discs, $i \in \{1,\ldots,n\}$, 
 of an augmented complex $\widetilde{C}\in DC(\widetilde{M})$, associated with $M$ that satisfy:

 \begin{enumerate}
  \item For each disc $i \in \{i, \ldots ,n\}$, ($\tau_i^*)$ is a permutation of the numbers $\{1, \ldots, n\}$, always starting with $1$,
  \item For all distinct disc indices $i_1,i_2$ and all edge indices $j_1,j_2$ such that  if two consecutive edge labels coincide, 
           $\tau_{i_1}^{j_1}=\tau_{i_2}^{j_2}$ and $\tau_{i_1}^{j_1+1}= \tau_{i_2}^{j_2+1}$,
           then at least one pair of corresponding matrix entries differs in sign, $M_{i_1,\tau_{i_1}^{j_1}}\neq M_{i_2,\tau_{i_2}^{j_2}}$ or $M_{i_1,\tau_{i_1}^{j_1+1}}\neq M_{i_2,\tau_{i_2}^{j_2+1}}$. 
           It is assumed that if $j_1$ is equal to $n$, then $j_1+1$ is equal to $1$, and the same for $j_2$.
 \end{enumerate}

\end{definition}

\begin{example}\label{Ex:sequence0}
The sequences $((1))$ and $((1,2),(1,2))$ are the unique valid sequences for $\h(1)$ and $\h(2)$, respectively. 
\end{example}

\begin{example}\label{Ex:sequence1}
 The sequence
 \begin{equation*}
  (\tau^j_i)=((1,3,2,4),(1,2,4,3),(1,3,2,4),(1,4,3,2))
 \end{equation*}
is a valid sequence for the Hadamard matrix

 $$\h(4)=
\begin{pmatrix*}[r]
1 & 1 & 1 & 1 \\
1 & -1 & 1 & -1 \\
1 & 1 & -1 & -1 \\
1 & -1 & -1 & 1 
\end{pmatrix*}.
$$
The first permutation (1,3,2,4) gives the ordering of the edges of the first disc that is associated to the first row of the matrix $\h (4)$.
As all entries of this first row are positive, we have all corresponding edges of the first disc with forward orientation.

In this example, the first and the third permutation are identical. In particular, the first parts (1,3) of the two permutations coincide,
but $\h(4)_{13}\neq \h(4)_{33}$. Another consecutive pair that appears in the second and the fourth permutation is (4,3), but 
$\h(4)_{24}\neq \h(4)_{44}$ and $\h(4)_{23}\neq \h(4)_{43}$.
If we compare all further consecutive edge pairs,
we can easily check that the given sequence of permutations is a valid sequence.
\end{example}

Using the inductive definition
\begin{equation*}
 \h(2n)=
\begin{pmatrix*}[r]
\h(n) & \h(n)\\
\h(n) & -\h(n)
\end{pmatrix*}
\end{equation*}
for the Hadamard matrices $\h(n)$, with the above sequence for $\h(1)$ as the base for the induction, we next provide a procedure 
to obtain a valid sequence $(\widetilde{\tau}_i^j$) for $\h(2n)$ from any valid sequence $(\tau_i^j)$ for $\h(n)$. 

For any fixed $i \in \{1, \ldots , n\}$, we create two permutations of the numbers $\{1, \ldots, 2n\}$ for $\h(2n)$, starting from the permutation $\tau_i^*$ of $\h(n)$:
\begin{align*}
 \beta_{i}^* = &(\tau^1_i,\ldots,\tau^n_i;\tau^1_i+n,\ldots,\tau^n_i+n), \\
 \gamma_{i}^* = &(\tau^1_i,\tau^2_i+n,\tau^3_i,\tau^4_i+n,\ldots, \tau^n_i+n;\tau^1_i+n,\tau^2_i,\tau^3_i+n,\tau^4_i,\ldots,\tau^n_i),
\end{align*}
i.e., $\beta_{i}^*$ is obtained from $\tau_i^*$ by first taking a copy of $\tau_i^*$ followed by another copy of $\tau_i^*$, where $n$ is added to each entry of the second copy;
and $\gamma_{i}^*$ is obtained by adding $n$ to all even positions of a copy of $\tau_i^*$ followed by doing so for all odd positions of a second copy of $\tau_i^*$.

\begin{proposition}\label{Prop:2}
 Given a valid sequence $(\tau_i^j)$ for $\h(n)$, the sequence $(\widetilde{\tau}_i^j)$, inductively defined by
\[
    \widetilde{\tau}_i^j= 
\begin{cases}
    \beta_{i}^j,& \text{if}\,\,\,\, i\leq n,\\
    \gamma_{i-n}^j, & \text{if}\,\,\,\, i> n,
\end{cases}
\ \ \ \mbox{\rm for}\,\,\,\, i,j \in \{1,\ldots , 2n\},
\]
is a valid sequence for $\h(2n)$.
\end{proposition}
In the definition of $\widetilde{\tau}_i^j$ in Proposition~\ref{Prop:2}, we first take all $n$ permutations $\beta_{i}^j$ of length $2n$
and then (after a shift of the row index by $n$) all $n$ permutations $\gamma_{i}^j$ of length $2n$ (in the given order, respectively).

\begin{example}\label{Ex:sequence2}
Starting with the valid sequence for $\h(2)$ from Example \ref{Ex:sequence0}, by the construction of Proposition~\ref{Prop:2} 
we obtain the following valid sequence for $\h(4)$:
\begin{align*}
  (\widetilde{\tau}_i^j)=&((1,2,3,4), (1,2,3,4), (1,4,3,2), (1,4,3,2)), 
 \end{align*}
which is different from the sequence of Example \ref{Ex:sequence1}, showing that valid sequences are not necessarily unique.
\end{example}

\begin{example}
 Starting with the valid sequence for $\h(4)$ from Example \ref{Ex:sequence2}, we obtain the following valid sequence for $\h(8)$:
 \begin{align*}
  (\widetilde{\tau}_i^j)=&((1,2,3,4,5,6,7,8), (1,2,3,4,5,6,7,8), (1,4,3,2,5,8,7,6), (1,4,3,2,5,8,7,6), \\
  & \mbox{} \hspace{1.75mm} (1,6,3,8,5,2,7,4), (1,6,3,8,5,2,7,4),  (1,8,3,6,5,4,7,2),(1,8,3,6,5,4,7,2)).
 \end{align*}
\end{example}

 \begin{proof}[Proof of Proposition \ref{Prop:2}]
  It is clear from the construction that for every $i$, $\widetilde{\tau}_i^*$ is a permutation of the numbers $\{1,\ldots, 2n\}$ starting with $1$, 
  so we only need to check the second condition to prove that $(\widetilde{\tau}_i^j)$ is a valid sequence for $\h(2n)$.
  
  The sequence $(\widetilde{\tau}_i^j$) for $\h(2n)$ is obtained from the valid sequence $(\tau_i^j)$ for $\h(n)$ in an inductive way.
  Let $i_1,i_2,j_1,j_2\in \{1,\ldots, 2n\}$, $i_1\neq i_2$, be such that $\widetilde{\tau}_{i_1}^{j_1}=\widetilde{\tau}_{i_2}^{j_2}$ and  $\widetilde{\tau}_{i_1}^{j_1+1}=\widetilde{\tau}_{i_2}^{j_2+1}$, 
  i.e.,  in the permutations corresponding to the two different discs $i_1,i_2$ two consecutive edge labels are the same.  
  Let $\hat{i}_1=i_1$ for $1\leq i_1\leq n$ and  $\hat{i}_1=i_1-n$ for $n+1\leq i_1\leq 2n$
  be the original disc indices in the initial valid sequence $(\tau_i^j)$ . 
  Same for $\hat{i}_2$ and analogously for $\hat{j}_1,\hat{j}_2$. It then follows by construction that $\tau_{\hat{i}_1}^{\hat{j}_1}=\tau_{\hat{i}_2}^{\hat{j}_2}$ 
  and $\tau_{\hat{i}_1}^{\hat{j}_1+1}=\tau_{\hat{i}_2}^{\hat{j}_2+1}$, as in the transition from $(\tau_i^j)$ to $(\widetilde{\tau}_i^j)$ either $n$ was added to the entries or not.
  There are then four different cases to consider, where we compare pairs of $\beta$- and $\gamma$-permutations:
  
  \begin{itemize}
   \item \mbox{[$\beta$--$\beta$]} $i_1,i_2 \leq n$: In this case $\h(2n)_{i_1,{\widetilde{\tau}_{i_1}^{j_1}}}=\h(n)_{\hat{i}_1,{\tau}_{\hat{i}_1}^{\hat{j}_1}}$, 
   and the same correlation holds for the other three matrix entries that appear in the definition of a valid sequence.  
   By assumption, $(\tau_{\hat{i}}^{\hat{j}})$ is a valid sequence for $\h(n)$, which implies
   $\h(n)_{\hat{i}_1,\tau_{\hat{i}_1}^{\hat{j}_1}}\neq \h(n)_{\hat{i}_2,\tau_{\hat{i}_2}^{\hat{j}_2}}$ or $\h(n)_{\hat{i}_1,\tau_{\hat{i}_1}^{\hat{j}_1+1}}\neq \h(n)_{\hat{i}_2,\tau_{\hat{i}_2}^{\hat{j}_2+1}}$.
   By the correlation between the entries of $\h(n)$ and $\h(2n)$ as written before, it follows that
   $\h(2n)_{i_1,\tau_{i_1}^{j_1}}\neq \h(2n)_{i_2,\tau_{i_2}^{j_2}}$ or $\h(2n)_{i_1,\tau_{i_1}^{j_1+1}}\neq \h(2n)_{i_2,\tau_{i_2}^{j_2+1}}$.

   \item \mbox{[$\gamma$--$\gamma$]} $i_1,i_2 >n$: This time $\h(2n)_{i_1,\widetilde{\tau}_{i_1}^{j_1}}=\pm \h(n)_{\hat{i}_1,{\tau}_{\hat{i}_1}^{\hat{j}_1}}$ with ``$+$'' for column indices 
   $1\leq \widetilde{\tau}_{i_1}^{j_1} \leq n$ and ``$-$'' otherwise. But this does not change anything, since $\widetilde{\tau}_{i_1}^{j_1}=\widetilde{\tau}_{i_2}^{j_2}$.
   Thus both $\h(2n)_{i_1,\widetilde{\tau}_{i_1}^{j_1}}$ and $\h(2n)_{i_2,\widetilde{\tau}_{i_2}^{j_2}}$ will keep the sign or change it,
    so we can use the same argument as before saying that since $(\tau_{\hat{i}}^{\hat{j}})$ was a valid sequence for $\h(n)$, we have a pair of matrix entries that are not equal.
   
   \item \mbox{[$\beta$--$\gamma$]} $i_1 \leq n$ and $i_2 >n$: As in the definition of $\beta_{i}^*$ and $\gamma_{i}^*$ for some permutation~$\tau_{i}^*$, the number $n$
   was added to all entries of the second half to obtain $\beta_{i}^*$, but added in an alternating fashion to the entries for obtaining $\gamma_{i}^*$, only the following 
   boundary cases are possible for which consecutive edge labels $\widetilde{\tau}_{i_1}^{j_1}=\widetilde{\tau}_{i_2}^{j_2}$ 
   and  $\widetilde{\tau}_{i_1}^{j_1+1}=\widetilde{\tau}_{i_2}^{j_2+1}$ can agree:
   positions $j_1=n$ or $j_1=2n$ in $\beta_{i}^*$ or $j_2=n$ or $j_2=2n$ in $\gamma_{i}^*$ . 
   If $j_1=n$, then $\widetilde{\tau}_{i_1}^{j_1+1}=n+\tau_{\hat{i}_1}^1=n+1=n+\tau_{\hat{i}_2}^1=\widetilde{\tau}_{i_2}^{n+1}$. In this case, the only possibility is that $j_2=n$, 
   but by construction, we have that $\widetilde{\tau}_{i_1}^{j_1}\neq\widetilde{\tau}_{i_2}^{j_2}$. The same argument can be used in the three other cases.
   
   \item \mbox{[$\gamma$--$\beta$]} $i_1 >n$ and $i_2 \leq n$: This case is analogous to the previous one.
  \end{itemize}
  
 In all four cases we showed that the second condition for having a valid sequence is satisfied, so we proved that $(\widetilde{\tau}_i^j)$ is a valid sequence for $\h(2n)$.
 \end{proof}

As a consequence of Proposition~\ref{Prop:2}  and the Examples~\ref{Ex:sequence0} and \ref{Ex:sequence1} we have that for every $n=2^k, \ k \geq 2$, 
there is a (not necessarily unique) valid sequence for the Hadamard matrix $\h(n)$. 

Via the recursive procedure of Proposition~\ref{Prop:2}, it takes quadratic time to obtain a valid sequence for $\h(n)$ from a valid sequence for $\h(\frac{n}{2})$,
as the $\left(\frac{n}{2}\right)^2$ numbers in the valid sequence for $\h(\frac{n}{2})$ have to be read and copied in modified form four times.
This gives us a recursive formula, $t(n)=cn^2+t(\frac{n}{2})$, for the total number of steps $t(n)$ to obtain a valid sequence for $\h(n)$ from the unique valid sequence for $\h(1)$.
It follows that for $n=2^k$, $t(n)\in O(\sum_{i=0}^k(2^i)^2)$, where $\sum_{i=0}^k(2^i)^2=2^{2k}\sum_{i=0}^k(2^{-i})^2$ gives us that $t(n)\in O(n^2)$.

\subsection{Triangulations HMT\mathversion{bold}$(n)$\mathversion{normal}  of the Hadamard examples H\mathversion{bold}$(n)$\mathversion{normal}}

Let $(\tau_i^j)$ be the valid sequence for the $(n \times n)$-Hadamard matrix $\h(n)=(\h_{ij})$ that is derived from the unique valid sequence of $\h(1)$
by the above procedure. 

From now on, we assume $n\geq 2$, as for $n=1$ there is only one disc, which can easily be triangulated as a single triangle
(hereby subdividing the boundary loop). Further, we will work with the augmented CW complex $\widetilde{C}(n)\in DC(\widetilde{\h}(n))$  
as defined earlier, i.e., in the polygonal disc corresponding to the row $i$ the order of the edges is given by $\tau_i^*$. 
In the following,  we will construct a triangulation $\mbox{\rm HMT}(n)$ of the complex $\widetilde{C}(n)$: 

\begin{figure}[t]
\begin{center}
\includegraphics[width=0.7\columnwidth]{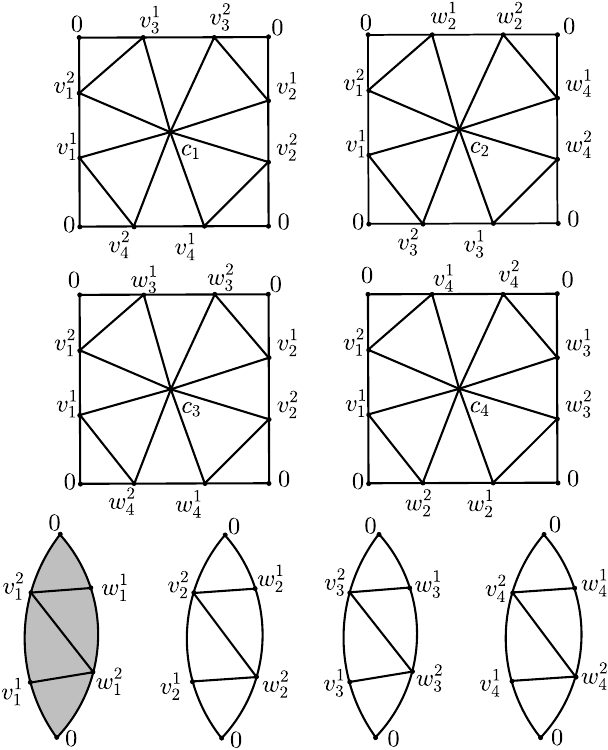}
\end{center}
\caption{The simplicial complex $\mbox{\rm HMT}(4)$ associated with the Hadamard matrix $\h(4)$.}
\label{Fig:4}
\end{figure}

\begin{itemize}
 \item For each column $j$ of $\h(n)$ we subdivide the two corresponding cycles $a_j^{+}$ and $a_j^{-}$ by using two additional vertices each, 
in particular, the subdivision of $a^{+}_j$ will be of the form $0$--$v^1_j$--$v^2_j$--$0$ and the subdivision of $a^{-}_j$ will be of the form $0$--$w^1_j$--$w^2_j$--$0$, 
using a total of $4n+1$ vertices for the $2n$ cycles and the initial vertex $0$.
 
 \item For each row $i$, we triangulate the interior of the corresponding disc (that has only positively oriented edges).  Each copy of the special vertex $0$ 
 is shielded off by a triangle that contains the respective copy and its two adjacent vertices, which together uses $n$ triangles.
Then a single additional vertex $c_i$ is placed at the center of the disc and connected to the $2n$ subdivision vertices of the loops
using $2n$ triangles. In total, this gives $3n$ triangles and a single additional vertex to triangulate one of the discs.

 \item We triangulate each of the $n$ digons between the loops $a_j^{+}$ and $a_j^{-}$ using the following four triangles:
 \begin{equation*}
  \{ \{0, v_j^1, w_j^2\}, \{v_j^1, v_j^2, w_j^2\}, \{v_j^2, w_j^1, w_j^2\}, \{0, v_j^2, w_j^1\} \}, \ \ \mbox{\rm for }\,\, j \in \{1, \ldots, n\}.
 \end{equation*}

\end{itemize}

The triangulation corresponding to the valid sequence 
described in Example~\ref{Ex:sequence2} for $\h(4)$ is drawn in Figure \ref{Fig:4}.

By the definition of a valid sequence, we easily see that the interiors of the $n$-gons do not share a single edge. This implies that  $\widetilde{\h}(n)$
is the boundary matrix of the first homology of our newly constructed simplicial complex. 

\begin{remark}\label{Rem:2}
 We actually never use the (negative) cycle $a_1^{-}$ in the construction of the $n$-gons, since for every $n$, the first column of $\h(n)$ has only $+1$'s. 
 We therefore delete from our construction the digon $a_1^{+}a_1^{-}$ (marked in grey in Figure \ref{Fig:4}) and save two vertices, $w_1^1$ and $w_1^2$, 
 without modifying the homotopy type of our complex. In total, we then need $5n-1$ vertices in our construction.
\end{remark}

Since in Section~\ref{Subsec:valid} we showed the existence of valid sequences and given that we know the determinant and the Smith Normal Form of the Hadamard matrices $\h(n)$,
 thanks to Lemma~\ref{Lemma:Smith}, we obtain the following: 

\begin{theorem}
 For each $n=2^k, \ k \geq 1$, there is a $\mathbb{Q}$-acyclic $2$-dimensional simplicial complex $\mbox{\rm HMT}(n)$ with face vector $f(\mbox{\rm HMT}(n))=(5n-1,3n^2+9n-6,3n^2+4n-4)$ and $H_*(\mbox{\rm HMT}(n))=(\Z, T(\mbox{\rm HMT}(n)), 0)$. The torsion in first homology is given by
 \begin{equation*}
  H_1(\mbox{\rm HMT}(n))=T(\mbox{\rm HMT}(n))=(\Z_2)^{{k}\choose{1}} \times (\Z_4)^{{k}\choose{2}}  \times \cdots \times (\Z_{2^k})^{{k}\choose{k}},
 \end{equation*}
where $|T(\mbox{\rm HMT}(n))|=n^{n/2} \in \Theta(2^{n\log n})$. Furthermore, the examples $\mbox{\rm HMT}(n)$ can be constructed algorithmically in quadratic time $\Theta(n^2)$.
\end{theorem}

\enlargethispage*{10mm}

\begin{proof}
 For any $k \geq 1$, due to Proposition \ref{Prop:2}, we have a valid sequence for $\h(n)$, where $n=2^k$. By Remark~\ref{Rem:2}, $\mbox{\rm HMT}(n)$ has $5n-1$ vertices. 
 For the edges, we have that each of the $2n-1$ cycles (not $2n$, because one is not used according to Remark \ref{Rem:2}) has three distinct edges. 
 By construction we are adding exactly $3n$ distinct edges in the interior of each of the $n$ polygonal discs and three additional edges for each of the $n-1$ digons, 
 yielding a total of $3n^2+9n-6$ edges. Again by construction, we are triangulating each of the $n$ polygonal discs with $3n$ triangles and each of the $n-1$ digons 
 with four additional triangles, which gives a total of $3n^2+4n-4$ triangles. 
 
 The simplicial complex $\mbox{\rm HMT}(n)$ is clearly connected, so $H_0(\mbox{\rm HMT}(n))=\Z$, and the statement about $H_1(\mbox{\rm HMT}(n))$ 
 follows by Lemma \ref{Lemma:Smith} on the Smith Normal Form of the Hadamard matrices $\h(n)$. The Euler characteristic of the simplicial complex 
 $\mbox{\rm HMT}(n)$ is~$1$ and we do not have a free part in the first homology, so we immediately obtain that $H_2(\mbox{\rm HMT}(n))=0$.
 
 Our construction of a valid sequence for $\h(n)$ takes quadratic time in $n$. Given the valid sequence, building the triangulation $\mbox{\rm HMT}(n)$
 takes linear time (in the size of the valid sequence), i.e., quadratic time in $n$ to output the triangulation with \mbox{$3n^2+4n-4$} triangles. Thus, in total,
 the examples $\mbox{\rm HMT}(n)$ are constructed in quadratic time $\Theta(n^2)$.
\end{proof}

An implementation \texttt{HMT.py} of our construction in \texttt{python} is available on GitHub at~\cite{HTPy}. Triangulations of the examples
$\mbox{\rm HMT}(2^k)$, $2\leq k\leq 5$, can be found online at the ``Library of Triangulations'' \cite{library}.

\begin{remark}
The construction above is not necessarily producing vertex-minimal triangulations of the disc complexes associated with the Hadamard matrices. 
For example, with the Random Simple Homotopy heuristic RSHT from~\cite{benedetti2021random} we were able to reduce the examples $\mbox{\rm HMT}(2)$, $\mbox{\rm HMT}(4)$, 
and $\mbox{\rm HMT}(8)$ to smaller triangulations with 6, 11, and 34 vertices, respectively.


A general strategy for a possible asymptotic reduction of the number of vertices needed is by Newman \cite{newman2019} via the \emph{pattern complex} for a given complex.
In particular, Newman used a revisited and randomized Speyer's construction that yields linearly many edges in the number 
of vertices for complexes associated with a given abelian group $G$ as their torsion group to obtain pattern complexes 
with $\Theta(\log (|G|)^{\frac{1}{2}})$ vertices---thus achieving Kalai's asymptotic growth $\Theta(2^{n^2})$ via a randomized construction. 
In our construction, however, the number of edges is quadratic in the number of vertices, and therefore the effect of taking
the pattern complex gives at most a linear improvement, which is not changing the asymptotic growth $\Theta(2^{n \log n})$ 
of the torsion size for our series of triangulations.
\end{remark}

\enlargethispage*{7.5mm}

\begin{remark}
We have been made aware (by an anonymous reviewer of an earlier conference submission of this paper) 
 of another folklore, but unpublished approach to achieve torsion that is exponential in $n^2$.
 The idea of that construction is to use Steiner systems $S(2,6,n)$. Such a system consists of a family of $6$-element subsets (called blocks) 
 of an $n$-element set, with the property that every pair of vertices is contained in exactly one block.
 On every block, one can construct a $6$-vertex triangulation of $\mathbb{R}P^2$. This gives a complex 
 with $\Theta(n^2)$ edge-disjoint $\mathbb{R}P^2$'s whose torsion is therefore $\Z_2^{\Theta(n^2)}$. 
 The existence of such systems has recently been proved by Keevash in \cite{keevash2018}, though in 
 a non-constructive way. A 
 deterministic construction that asymptotically yields high torsion can still be obtained, 
 using the fact that it is enough to build on a partial system containing $\Theta(n^2)$ $6$-element blocks 
 where each pair is contained in \emph{at most} one block. 
 These can be generated 
 via a polynomial-time derandomization of Rodl's Nibble~\cite{grable1996nearly}, though the procedure 
 seems to be impractical and we are not aware of an implementation.
\end{remark}

\bibliographystyle{plainurl}

\small
\bibliography{bibliography}

\end{document}